\documentclass[12pt,a4paper]{amsart}
\usepackage[a4paper,left=2.35cm,right=2.35cm,top=3cm,bottom=3cm]{geometry}

\usepackage{graphicx} % Required for inserting images
\usepackage{xcolor}
\usepackage{amsfonts,stmaryrd}
\usepackage{amsthm}
\usepackage{amsmath}
\usepackage{amssymb}
\usepackage{enumerate}
\usepackage{geometry}
\usepackage{tabularx}
\usepackage{multicol}
\usepackage{ytableau}
\usepackage{adjustbox}
\usepackage{hyperref}
\usepackage{mathdots}
\usepackage{euler}

\hypersetup{
    colorlinks = true,
    citecolor = {blue},
}
\usepackage{tikz}
\usetikzlibrary{cd}
\usepackage{array}
\newcolumntype{L}[1]{>{\raggedright\let\newline\\\arraybackslash\hspace{0pt}}m{#1}}
\newcolumntype{C}[1]{>{\centering\let\newline\\\arraybackslash\hspace{0pt}}m{#1}}
\newcolumntype{R}[1]{>{\raggedleft\let\newline\\\arraybackslash\hspace{0pt}}m{#1}}

\newcommand{\g}{\mathfrak{g}}

\newcommand{\N}{\mathbb{N}}
\newcommand{\Z}{\mathbb{Z}}

\newcommand{\wt}{\mathrm{wt}}

\newtheorem{thm}{Theorem}[section]
\newtheorem{proposition}[thm]{Proposition}
\newtheorem{conjecture}[thm]{Conjecture}
\newtheorem{lemma}[thm]{Lemma}
\newtheorem{corollary}[thm]{Corollary}
\newtheorem{definition}[thm]{Definition}

\newtheorem{example}[thm]{Example}

\title{Cacti, Toggles, and Reverse Plane Partitions}

\author{Devin Brown}
\author{Bal\'azs Elek}
\author{Iva Halacheva}
\address{Devin Brown, Iva Halacheva, Department of Mathematics, Northeastern University}
\email{brown.de@northeastern.edu}
\email{i.halacheva@northeastern.edu}
\address{Bal\'azs Elek, Department of Mathematics, University of British Columbia}\email{balazse@math.ubc.ca}
\thanks{Devin Brown was supported by Northeastern's AJC Merit Research Scholarship and PEAK Summit Award.\\ Iva Halacheva was supported by NSF grant  DMS-2302664.
}


\begin{document}
\begin{abstract}
The cactus group acts combinatorially on crystals via partial Sch{\"u}tzenberger involutions. This action has been studied extensively in type $A$ and described via Bender-Knuth involutions. We prove an analogous result for the family of crystals $B(n\varpi_1)$ in type $D$. Our main tools are combinatorial toggles acting on reverse plane partitions of height $n$. As a corollary, we show that the length one and two subdiagram elements generate the full cactus action, addressing conjectures of Dranowski, the second author, Kamnitzer, and Morton-Ferguson.
\end{abstract}

\maketitle

\section{Introduction}

The crystals for a simple Lie algebra $\g$ are combinatorial shadows of its representations that lend themselves to techniques from algebraic combinatorics. There is a rich history of the development of such techniques especially in type $A$. For instance, the Bender--Knuth (BK) involutions on semistandard Young tableaux (SSYT) were originally employed in \cite{BK72} to count different classes of plane partitions, and have since found many applications, for instance providing an elegant proof of the Littlewood-Richardson rule \cite{St02}. The relations between these involutions were studied in \cite{BK95}, as the Berenstein--Kirillov group of BK involutions acting on SSYT, or equivalently on Gelfand-Tsetlin patterns. The cactus group $C_\g$, whose connection to crystals was first studied in \cite{HeKa}, has the advantage of having a type-independent description. In type $A$, it was shown (\cite{CGP}, \cite{Hal}) that the Berenstein--Kirillov group is a quotient of $C_\g$. More recently, in \cite{Dra} the authors used \textit{heaps} and \textit{reverse plane partitions (RPPs)} to provide a type-indendent model for minuscule Demazure crystals in $ADE$ types, as well as a group of combinatorial \textit{toggles} acting, that specialize to the BK involutions in type $A$. A natural question that emerges is, can one relate these toggle operators to the cactus group beyond type A and further extend them to more general (not necessarily minuscule) crystals. In particular, for a minuscule dominant weight $\lambda$ and a corresponding Weyl group element $w^J_0$, does the cactus action factor through the toggle action, as follows:

\begin{conjecture}[\cite{Dra}]\label{cong:cactus toggle}\hfill
    \begin{enumerate}
        \item There is a surjective map $C_{\g} \to Tog(w_0^J)$ such that the action of the cactus group on the crystal $B(n\lambda) \cong RPP(w_0^J, n)$ factors through this map.
        \item In type $D_m$, the action of the cactus group on the crystal $B(n\varpi_1)$ is generated by cactus elements corresponding to length $1$ and $2$ subdiagrams of the Dynkin diagram.
    \end{enumerate}
\end{conjecture}

In this paper, we prove both parts of Conjecture \ref{cong:cactus toggle} for the crystal $B(n\varpi_1)$ in type $D_m$. In particular, we give an explicit description of the cactus operators in terms of toggles. This reduces the action to significantly fewer cactus generators than originally required.

Further questions of interest are connecting the cactus and toggle actions in type D to promotion in algebraic combinatorics. A potential application of our result is a simplified description of the affine crystal operators $e_0$ and $f_0$ for the Kirillov-Reshetikhin crystals $B^{1,n}$ for (affine) type $D_m^{(1)}$, using the cactus action description through toggles which are much easier to compute on $B(n\varpi_1)$. To complete the study of minuscule weights in types ADE, one needs to consider the two spin nodes in type $D_m$, two more cases in $E_6$, and one case in $E_7$.

\section{Crystals and Tensor Products}

Throughout, we let $\g$ denote a simple complex Lie algebra with Dynkin diagram $I$, simple roots $\alpha_i$, Weyl group $W$ generated by the simple reflections $s_i$, $i \in I$, weight lattice $P$, and dominant weights $P_+$ (see for instance \cite{Hum}). We will focus on the Lie algebra of type $D_m$, where $I$ is labeled such that $m-1$ and $m$ denote the spin nodes. 

Crystal bases were originally introduced by Kashiwara \cite{Kas90, Kas91}, as corresponding to bases for representations of the quantum group $U_q(\g)$ in a certain limit as $q$ goes to $0$. Although there is a more general notion of abstract crystals, we will only consider crystals arising from representations (see for instance  \cite[Section 4.5]{HK} for more details). A \textbf{(normal) $\g$-crystal} is a finite set $B$ together with maps $\mathrm{wt}\colon B\to P$, $e_i\colon B \to B\sqcup \{0\}$ (the raising operators), $f_i\colon B \to B\sqcup \{0\}$ (the lowering operators)  for each $i\in I$, as well as $\varepsilon_i, \phi_i:B \rightarrow \Z$, describing $i$-string lengths, defined by $\phi_i(b) = \max \{n \mid f_i^n(b)\neq 0\}$ and $\varepsilon_i(b) = \max\{n \mid e_i^n(b)\neq 0\}$ which satisfy several conditions coming from the associated representation of $\g$. Notably, the raising and lowering operators are partial inverses: if $b_1, b_2 \in B$, then $e_i(b_1) = b_2$ if and only if $f_i(b_2) = b_1$.  We focus on $\g$-crystals whose highest weight is minuscule or a multiple of a minuscule weight.

\begin{definition}
    A representation $V$ of $\g$ is called \textbf{minuscule} if $W$ acts transitively on the weights of $V$. A weight $\lambda\in P_+$ is called \textbf{minuscule} if $V(\lambda)$ is minuscule.
\end{definition}

In type $A_m$, the minuscule weights are exactly all the fundamental weights $\varpi_k = \epsilon_1 + \cdots + \epsilon_k, \; 1 \leq k \leq m$. In type $D_m$, the minuscule weights are $\varpi_1$, $\varpi_{m-1}$, and $\varpi_m$.

\begin{example}\label{eg:mincrys}
Using the generalized semistandard Young tableaux (SSYT) model for representations of the Lie algebra of type $D_m$ (see \cite[Section 8.5]{HK}), we have the partially ordered set (poset, in short) $\mathbf{N}$ of possible fillings, $1 < 2 < \cdots < m-1 < m,\overline{m} < \overline{m-1} < \cdots < \overline{2} < \overline{1}$ where $m$ and $\overline{m}$ are incomparable. The crystal $B(\varpi_1)$ can be realized as shown in Figure \ref{fig:Standard-crystal}, where the box with filling $i$ has weight $\epsilon_i$, and $\overline{i}$ has weight $-\epsilon_i$. 
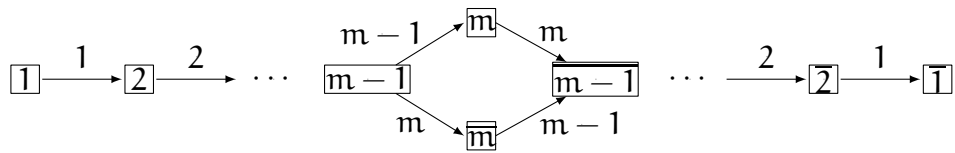
\begin{figure}[h!]
    \centering
    \begin{tikzpicture}[>=latex,line join=bevel, scale = 0.75]
        \draw[black, thin] (-0.25,-0.25) rectangle (0.25,0.25);
        \draw[black, thin] (1.75,-0.25) rectangle (2.25,0.25);
        \draw[black, thin] (5.25, -0.25) rectangle (6.75, 0.25);
        \draw[black, thin] (7.75, 0.75) rectangle (8.25, 1.25);
        \draw[black, thin] (7.75, -1.25) rectangle (8.25,-0.75);
        \draw[black, thin] (9.25,-0.3) rectangle (10.75,0.3);
        \draw[black, thin] (13.75,-0.25) rectangle (14.25,0.25);
        \draw[black, thin] (15.75,-0.25) rectangle (16.25,0.25);
        \filldraw[black] (0,0) circle (0pt) node[anchor=center]{$1$};
        \filldraw[black] (2,0) circle (0pt) node[anchor=center]{$2$};
        \filldraw[black] (4.35,0) circle (0pt) node[anchor=center]{$\cdots$};
        \filldraw[black] (6,0) circle (0pt) node[anchor=center]{$m-1$};
        \filldraw[black] (8,1) circle (0pt) node[anchor=center]{$m$};
        \filldraw[black] (8,-1) circle (0pt) node[anchor=center]{$\overline{m}$};
        \filldraw[black] (10,0) circle (0pt) node[anchor=center]{$\overline{m-1}$};
        \filldraw[black] (11.65,0) circle (0pt) node[anchor=center]{$\cdots$};
        \filldraw[black] (14,0) circle (0pt) node[anchor=center]{$\overline{2}$};
        \filldraw[black] (16,0) circle (0pt) node[anchor=center]{$\overline{1}$};
        \draw[black, thin, ->] (0.30,0) -- (1.70,0);
        \draw[black, thin, ->] (2.30,0) -- (3.70,0);
        \draw[black, thin, ->] (6.5,0.25) -- (7.70,1);
        \draw[black, thin, ->] (6.5,-0.25) -- (7.70,-1);
        \draw[black, thin, ->] (8.25,1) -- (9.5,0.3);
        \draw[black, thin, ->] (8.25,-1) -- (9.5,-0.3);
        \draw[black, thin, ->] (12.30,0) -- (13.70,0);
        \draw[black, thin, ->] (14.30,0) -- (15.70,0);
        \filldraw[black] (1,0.4) circle (0pt) node[anchor=center]{$1$};
        \filldraw[black] (3,0.4) circle (0pt) node[anchor=center]{$2$};
        \filldraw[black] (6.25,0.8) circle (0pt) node[anchor=center]{$m-1$};
        \filldraw[black] (6.75,-0.8) circle (0pt) node[anchor=center]{$m$};
        \filldraw[black] (9.75,-0.8) circle (0pt) node[anchor=center]{$m-1$};
        \filldraw[black] (9.25,0.8) circle (0pt) node[anchor=center]{$m$};
        \filldraw[black] (13,0.4) circle (0pt) node[anchor=center]{$2$};
        \filldraw[black] (15,0.4) circle (0pt) node[anchor=center]{$1$};
    \end{tikzpicture}
    \caption{The crystal $B(\varpi_1)$ in type $D_m$.}
    \label{fig:Standard-crystal}
\end{figure}
\end{example}

Crystals have a particularly nice tensor product rule, which we recall next. Note that we use a different convention from Kashiwara's convention (as in \cite{HK}), instead following \cite[Section 2.3]{BS} and the tensor product rule implemented in SageMath.

\begin{definition} If $B_1$ and $B_2$ are $\g$-crystals, then one can define a $\g$-crystal structure on the set
    $B_1\otimes B_2 = \{b_1 \otimes b_2 \mid b_1\in B_1, b_2\in B_2\}$ with $\wt(b_1 \otimes b_2) = \wt(b_1) + \wt(b_2)$, and the following lowering operators
    \begin{align*}
        f_i(b_1\otimes b_2) = \begin{cases}
            b_1 \otimes f_i(b_2) & \text{ if } \varepsilon_i(b_1) < \phi_i(b_2)\\
            f_i(b_1) \otimes b_2 & \text{ if } \varepsilon_i(b_1) \ge \phi_i(b_2).
        \end{cases}
    \end{align*}
\end{definition}
  
This rule extends to $B^{\otimes n}$ for any $\g$-crystal $B$ as follows.

\begin{proposition}[\cite{Tin08}, Corollary 2.4]
    For any $b = b_1\otimes \cdots \otimes b_n \in B^{\otimes n}$, define the sign pattern of $b$ to be the following string of $+$'s and $-$'s: $\phi_i(b_1)$ many $+$'s, $\varepsilon_i(b_1)$ many $-$'s, $\phi_i(b_2)$ many $+$'s, $\varepsilon_i(b_2)$ many $-$'s, etc. Pair up and cancel substrings $-+$, to get a sequence $+\hdots + - \hdots -$. If $j$ is the index corresponding to the rightmost unpaired $+$ and $k$ is the index of the leftmost unpaired $-$, then $f_i(b) = b_1 \otimes \cdots \otimes f_i(b_j) \otimes \cdots \otimes b_n$ and $e_i(b) = b_1 \otimes \cdots \otimes e_i(b_k) \otimes \cdots \otimes b_n$. We have $f_i(b) = 0$ if there is no such $j$, and $e_i(b) = 0$ if there is no such $k$.
\end{proposition}

\begin{example}\label{eg:multmincrys}
Continuing Example \ref{eg:mincrys}, one can realize more generally the crystal $B(n\varpi_1)$ as the collection of tableaux with shape $1\times n$ and fillings in $\mathbf{N}$ that are weakly increasing from left to right. In particular, $m$ and $\overline{m}$ do not appear simultaneously. The raising and lowering operators on this set can be defined via the embedding $B(n\varpi_1)\hookrightarrow B(\varpi_1)^{\otimes n}$, given by {$$\small\begin{ytableau} b_1 & b_2 & \cdots & b_n
\end{ytableau} \stackrel{\mapsto}{\phantom{a}} \begin{ytableau}
    b_1
\end{ytableau} \stackrel{\otimes}{\phantom{a}}  \begin{ytableau}
    b_2
\end{ytableau} \stackrel{\otimes}{\phantom{a}}  \cdots \stackrel{\otimes}{\phantom{a}} 
\begin{ytableau}
        b_n
\end{ytableau},$$} and are determined by the tensor product rule on $B(\varpi_1)^{\otimes n}$.
\end{example}

\section{Heaps and Reverse Plane Partitions}

 Heaps, alongside tableaux, are the key combinatorial structure we will use to study crystals. We start by recalling their construction and properties (see \cite{St96,Pro,St01}).

\begin{definition}
    Given a reduced word $\mathbf{w} = (s_{i_1}, \dots, s_{i_l})$ for $w\in W$, the corresponding \textbf{heap} $H(\mathbf{w})$ is the poset $(\{1,2,\dots, l\},\prec)$ defined as the transitive closure of the relation: $a \prec b$ if $a>b$ and the nodes $i_a$ and $i_b$ are adjacent in the Dynkin diagram I.
\end{definition}

    Using the abacus analogy due to A. Kleshchev and A. Ram \cite{KR}, one can construct a heap by starting with a base the shape of I, and placing beads corresponding to the reflections in the reduced word for $\mathbf{w} = (s_{i_1}, \dots, s_{i_l})$, read from right to left, onto the corresponding runners. The idea is that beads placed above adjacent vertices will be prevented from falling to the bottom of the abacus as pictured in Figure \ref{fig:heaps}.

We can define a map $\pi\colon H(\mathbf{w}) \to I$ that projects a bead onto its corresponding runner, i.e. the reflection $s_i$ maps to the vertex $i$. The condition that $\mathbf{w}$ is a reduced word for $w$ tells us that $\pi^{-1}(i)$ is totally ordered for each $i\in I$, since otherwise there would be two occurrences of $s_{i}$ that are not separated by adjacent reflections, allowing a cancellation. 

It can be the case that two reduced words for an element of $W$ give rise to non-isomorphic heaps.
    
\begin{example}
    Consider the reduced words $\mathbf{w} = (s_1, s_3, s_2, s_1)$ and $\mathbf{w^\prime} = (s_3, s_2, s_1, s_2)$ in type $A_3$. Then $\mathbf{w}$ and $\mathbf{w^\prime}$ are both reduced words for the same element of $W$. The heaps $H(\mathbf{w})$ and $H(\mathbf{w^\prime})$ are not isomorphic since $H(\mathbf{w^\prime})$ is totally ordered, whereas $H(\mathbf{w})$ is not (see Figure \ref{fig:heaps}).
\end{example}
    
    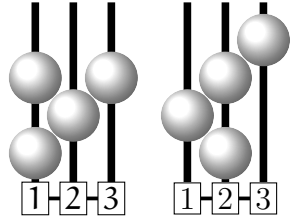
\begin{figure}[h]
    \centering
    \begin{tikzpicture}[scale = 0.5]
    \foreach \x in {0,1,2} 
      \draw [line width=1mm] (\x, 0) -- (\x,5.2);
    \foreach \x in {0,1,2} 
      \shade[shading=ball, ball color=white] (\x,\x+1.2) circle (0.7);
    \foreach \x in {0} 
      \shade[shading=ball, ball color=white] (\x,\x+3.2) circle (0.7);
    \draw [line width=0.5mm] (0,0) -- (2,0);
    \foreach \x/\xtext in {0/1,1/2,2/3} 
      \draw (\x cm,0 pt) node[fill=white,inner sep=2pt,draw] {${\xtext}$};
      
    \foreach \x in {0,1,2} 
      \draw [line width=1mm] (\x+4, 0) -- (\x+4,5.2);
    \foreach \x in {0,1,2} 
      \shade[shading=ball, ball color=white] (\x+4,\x+2.2) circle (0.7);
    \foreach \x in {0} 
      \shade[shading=ball, ball color=white] (\x+5,\x+1.2) circle (0.7);
    \draw [line width=0.5mm] (4,0) -- (6,0);
    \draw (4 cm,0 pt) node[fill=white,inner sep=2pt,draw] {\text{1}};
    \draw (5 cm,0 pt) node[fill=white,inner sep=2pt,draw] {\text{2}};
    \draw (6 cm,0 pt) node[fill=white,inner sep=2pt,draw] {\text{3}};
  \end{tikzpicture}
  \caption{$H(\mathbf{w})$ and $H(\mathbf{w^\prime})$ in type $A_3$.}
  \label{fig:heaps}
\end{figure}

\begin{definition}
    An element $w\in W$ is called \textbf{fully commutative} if any two reduced words for $w$ can be obtained from each other only by commuting elements, e.g. $s_2s_1s_3s_2 = s_2s_3s_1s_2$ in $A_3$.
\end{definition}

If $w$ is fully commutative, then any two reduced words for $w$ give rise to isomorphic posets, so we may refer to the heap $H(w)$. We next recall a generalization of such heaps that captures the structure of crystals with highest weight a multiple of a minuscule weight. We follow the exposition and notation in \cite[Section 2]{Dra}.

\begin{definition} A \textbf{reverse plane partition} (RPP) of shape $H(w)$ and height $n\in \N$ is a function $\Phi: H(w) \to \{0,1,\dots, n\}$ such that if $x\preceq y$, then $\Phi(x) \ge \Phi(y)$. We denote the collection of reverse plane partitions by $\text{RPP}(w,n)$.
\end{definition}

Using the abacus analogy, RPPs correspond to labelling the beads in a heap with integers from $0$ to $n$ so that they decrease from the base to the top. We can also represent RPPs as increasing chains of order ideals.

\begin{definition}
    If $(\mathcal{P}, \preceq)$ is a poset, then we call a subset $S\subseteq \mathcal{P}$ an \textbf{order ideal} of $\mathcal{P}$ if for all $x\in S$ and $y\in \mathcal{P}$, if $y\preceq x$, then $y\in S$. Let $J(\mathcal{P})$ denote the collection of order ideals of $\mathcal{P}$.
\end{definition}

Noting that $J(\mathcal{P})$ is also a poset ordered under inclusion, we have the following lemma embedding $RPP(w,n)$ in $J(H(w))^n$.

\begin{lemma}[\cite{Dra}, Lemma 2.28]\label{lem:Increasing Chain}
    There is an injective map $RPP(w,n) \hookrightarrow J(H(w))^n$ given by $\Phi \mapsto (\phi_1, \dots, \phi_n)$ where $\phi_i$ is defined as the pre-image $\Phi^{-1}(\{n-i+1, \dots, n\})$. The image of this map is the collection of tuples $\{(\phi_1, \dots, \phi_n) \in J(H(w))^n \mid \forall \; i, \; \phi_i \subseteq \phi_{i+1}\}$.
\end{lemma}

In order to read off the RPP from a chain of order ideals $(\phi_1, \dots, \phi_n)$, one labels the elements of $\phi_1$ with $n$, the elements of $\phi_2$ not included in $\phi_1$ with $n-1$, and so on, labeling the elements of $H(w)$ not included in $\phi_n$ with $0$.

\begin{example}\label{ex:RPP}
    If $w = s_3s_4s_2s_1s_3s_2\in W$ in type $A_4$, then $w$ is fully commutative. Figure \ref{fig:RppExample} depicts an RPP $\Phi$ of shape $H(w)$ and height $3$ and the corresponding chain of order ideals, where we identify an order ideal with its indicator function in $RPP(w,1)$.
    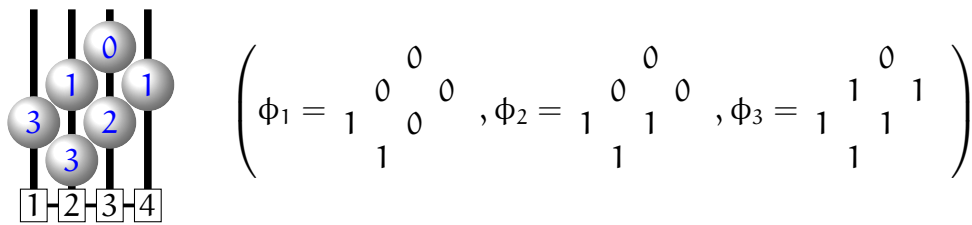
\begin{figure}[h]
        \centering
        \begin{tikzpicture}[scale = 0.5]
            \foreach \x in {0,1,2,3} 
                \draw [line width=1mm] (\x, 0) -- (\x,5.2);
            \foreach \x in {0,1,2} 
                \shade[shading=ball, ball color=white] (\x,\x+2.2) circle (0.7);
            \foreach \x in {0,1,2} 
                \shade[shading=ball, ball color=white] (\x+1,\x+1.2) circle (0.7);
            \draw [line width=0.5mm] (0,0) -- (3,0);
            \foreach \x/\xtext in {0/1,1/2,2/3,3/4} 
                \draw (\x cm,0 pt) node[fill=white,inner sep=2pt,draw] {${\xtext}$};
            \draw (1,1.2) node[fill=none] {\textcolor{blue}{$3$}};
            \draw (0,2.2) node[fill=none] {\textcolor{blue}{$3$}};
            \draw (2,2.2) node[fill=none] {\textcolor{blue}{$2$}};
            \draw (1,3.2) node[fill=none] {\textcolor{blue}{$1$}};
            \draw (2,4.2) node[fill=none] {\textcolor{blue}{$0$}};
            \draw (3,3.2) node[fill=none] {\textcolor{blue}{$1$}};
            \draw (15,2.5) node[fill=none] {\textcolor{black}{$\left (\phi_1 = \arraycolsep=3pt\def\arraystretch{0.9}\begin{array}{ccccccc}
	& & 0 & & \\
	& 0 & & 0 & \\
	1 & & 0 & & \\
	& 1 & & & 
	\end{array}, \phi_2 = \arraycolsep=3pt\def\arraystretch{0.9}\begin{array}{ccccccc}
	& & 0 & & \\
	& 0 & & 0 & \\
	1 & & 1 & & \\
	& 1 & & & 
	\end{array}, \phi_3 = \arraycolsep=3pt\def\arraystretch{0.9}\begin{array}{ccccccc}
	& & 0 & & \\
	& 1 & & 1 & \\
	1 & & 1 & & \\
	& 1 & & & 
	\end{array}			
    \right )$}};
        \end{tikzpicture}
            
        \caption{An element $\Phi$ of $RPP(w, 3)$ and the corresponding chain of order ideals.}
        \label{fig:RppExample}
    \end{figure}
\end{example}

\section{Minuscule crystals as heaps}

The structure of the crystal $B(\lambda)$ for a minuscule weight $\lambda$ can be combinatorially realized through heaps via the following construction. Note that since we work with $\g$ a simple Lie algebra, then all the minuscule highest weights are fundamental. Consider $\lambda=\varpi_j \in P_+$, and $J=I\setminus\{j\}$. Then $W_J=\left\langle s_j~:~j \in J\right\rangle \leq W$ is the stabilizer of $\lambda$, and let $W^J$ denote the set of minimal length representatives of the cosets $W/W^J$. In particular, let $w_0$ denote the longest element in $W$. If $w^J_0$ denotes the minimal length representative of $w_0W_J$, then it is of minimal length such that $w^J_0\lambda=w_0\lambda$, and $W^J = \{v\in W\colon v\le_L w_0^J\}$ where $v\le_L w$ is the \textit{weak left order} (i.e. there exist reduced expressions such that $v$ appears as a terminal subword of $w$).

\begin{thm}[\cite{St96}, Lemma 3.1, Theorem 6.1, Theorem 7.1]\hfill
\begin{enumerate}
    \item If $w \in W$ is fully-commutative, then $\{v\in W \colon v\le_L w\} \cong J(H(w))$ as posets.
    \item For $\lambda$, $J$ as above, $\lambda$ is minuscule if and only if every element of $W^J$ is fully commutative. 
\end{enumerate}
\end{thm}

Stembridge's result above in fact holds in greater generality, where $\g$ can be semisimple, and $\lambda$ a sum of fundamental weights. In the case where $\lambda=\varpi_j$ is minuscule, the weight spaces of $V(\lambda)$ are $1$-dimensional and hence the orbit $W\lambda$ is in bijection with the set of coset representatives $W^J=\{v\in W \colon v\le_L w^0_J\} \cong J(H(w^0_J))$, for $J=I\setminus\{j\}$. We can also describe the local shape of the heap of such minuscule elements as follows (this is due to Proctor in the simply laced case \cite{Pro}).

\begin{proposition}[\cite{St01}, Proposition 3.3]\label{prop:DmHeap}
    In the heap of a minuscule element, every closed subinterval between two elements labeled $i$, which have no elements labeled $i$ between them, is isomorphic as a labeled poset to the heap of $s_1\cdots s_{m-2}s_{m-1}s_ms_{m-2}\cdots s_1$ in type $D_m$ for $m\ge 3$.
\end{proposition}

\begin{proposition}[\cite{Dra}, Proposition 2.18]\label{prop:Coset model}
    The set $W^J$ is a model for the crystal $B(\lambda)$ with weight map given by $\wt(v) = v\lambda$ and the following raising and lowering operators.
    \begin{gather*}
  f_i(v)=
  \begin{cases}
    s_i v &\text{ if } s_i v > v  \text{ and } s_i v \in W^J \\
    0 &\text{ otherwise,}
  \end{cases}
  \qquad
  e_i(v)=
  \begin{cases}
    s_i v &\text{ if } s_i v < v \text{ and } s_i v \in W^J\\
    0 &\text{ otherwise.}
  \end{cases}
\end{gather*}  
\end{proposition}

\begin{example}\label{ex:D4 Heap}
    For $\lambda = \varpi_1$ in type $D_4$, we have $J = \{2,3,4\}$ and $w_0^J = s_1s_2s_3s_4s_2s_1$.
\end{example}

This crystal model can be generalized to $\g$-crystals with highest weight $n \lambda$ where $\lambda$ is minuscule, using $RPP(w,n)$ and the tensor product rule on $J(H(w))^n$:

\begin{thm}[\cite{Dra}, Theorem 2.29]\label{thm:RPP Crystal}
    Let $\lambda =\varpi_j$ and $J=I\setminus\{j\}$. If $\lambda$ is minuscule, then there is an isomorphism of crystals $B(n\lambda) \cong RPP(w_0^J, n)$ compatible with the inclusions $B(n\lambda) \hookrightarrow B(\lambda)^{\otimes n}$ and $RPP(w^J_0,n) \hookrightarrow J(H(w^J_0))^n$.
\end{thm}

\section{Groups Acting on Crystals}

There are several groups generated by involutions that act on crystals. In type $A$, the Berenstein--Kirillov group acts on semistandard Young tableaux and Gelfand-Tsetlin patterns, alternative models for crystals \cite{BK95}. We consider the action of the cactus group and toggles in a more general setting. 

\subsection{The cactus group}

To define the cactus group corresponding to the Lie algebra $\g$, we first recall the Dynkin diagram automorphisms $\theta_J:J\rightarrow J$, where $J\subseteq I$ is any a connected subdiagram, determined by $\alpha_{\theta_J(j)}=-w_{0,J}(\alpha_j)$, for $w_{0,J}$ the longest element of the subgroup $W_J$. In type $A_m$ this automorphism is given by reflecting, $\theta_I(i) = m-i+1$, and in type $D_m$, $\theta_I$ is the identity when $m$ is even, and swaps the spin nodes $m-1$ and $m$ when $m$ is odd. 

\begin{definition}
    The \textbf{cactus group} $C_{\g}$ is the group with generators $c_J$ for each connected subdiagram $J\subseteq I$, and relations $c_J^2 = 1$, $c_Jc_K = c_Kc_J$ if $J\cup K$ is disconnected, $c_Jc_K = c_Kc_{\theta_K(J)}$ if $J\subseteq K$, and no relations between $c_J$ and $c_K$ otherwise.
\end{definition}

Given the cactus relations, a smaller set of generators is often sufficient:

\begin{proposition}
    The cactus group $C_{D_m}$ is generated by elements corresponding to subdiagrams of type $D$, $c_{\{j,\dots,m-1,m\}}$ for $1\le j\le m-2$, and elements corresponding to subdiagrams of type $A$, $c_{\{1,\dots, j\}}$ for $1\le j\le m-1$, as well as the additional subdiagram $c_{\{1,\dots, m-2,m\}}$ if $m$ is even.
\end{proposition}

The name \textit{cactus} comes from identifying $C_{A_m}$ with the fundamental group of the moduli space of real genus $0$ stable curves with marked points \cite{HeKa}. The (set) action of this group on any $\g$-crystal is realized via generalizations of the classical Sch\"utzenberger involution in algebraic combinatorics. 

\begin{definition}
    For a highest weight irreducible $\g$-crystal $B(\lambda)$, the \textbf{Sch{\"u}tzenberger involution} is the unique map of sets $\xi\colon B(\lambda)\to B(\lambda)$ which satisfies the following conditions: $$ \wt(\xi(b)) = w_0 \wt(b), \quad  e_i\cdot \xi(b) = \xi(f_{\theta_I(i)}(b)), \quad  f_i\cdot \xi(b) = \xi(e_{\theta_I(i)}(b)).$$ It can be extended to any $\g$-crystal $B=\bigsqcup_\lambda B(\lambda)$ by acting as above on each of the components.
\end{definition}

The effect of applying $\xi$ is that the highest and lowest weights of $B(\lambda)$ are swapped.

\begin{thm}[\cite{HKRW}, Theorem 5.19]
    There is an action of the cactus group $C_{\g}$ on any $\g$-crystal $B$ given by, for any connected $J\subseteq I$,
   $c_J(b) = \xi_{B_J}(b)$, 
    where $B_J$ is the restriction of $B$ to the subdiagram $J$, which consists of the same vertices and only edges labeled by $j\in J$.
\end{thm}

\subsection{Toggles}

Another group that acts on crystals is the group of so-called combinatorial toggles which are defined for order ideals of a poset and can be extended to reverse plane partitions.

\begin{definition}
    Let $(\mathcal{P}, \preceq)$ be a poset. Then for any $x\in \mathcal{P}$, define the associated \textbf{combinatorial toggle} $t_x\colon J(\mathcal{P}) \to  J(\mathcal{P})$ by
    \begin{align*}
        t_x(S) &= \begin{cases}
            S \cup \{x\} & \text{ if } x\notin S \text{ and } S\cup \{x\}\in J(\mathcal{P})\\
            S\backslash \{x\} & \text{ if } x\in S \text{ and } S\backslash \{x\}\in J(\mathcal{P})\\
            S & \text{ otherwise}
        \end{cases}.
    \end{align*}
\end{definition}
In other words, the toggle associated to $x$ either adds or removes $x$ depending on whether the result is a valid order ideal. In the case where $\mathcal{P} = H(w)$ for some fully commutative $w\in W$, we have that the toggles $t_x$ and $t_y$ commute whenever $\pi(x) = \pi(y)$, since each fiber $\pi^{-1}(i)$ is totally ordered. Hence   
$t_i = \prod_{x\in \pi^{-1}(i)} t_x$
is a well defined function $J(H(w)) \to J(H(w))$. In the bead analogy, these elements correspond to either adding or removing a bead on the $i$-th runner whenever legal. The \textbf{toggle group} $Tog(w)$ is the group generated by $t_i$ for each $i\in I$.

The action of the toggle group can be expended to reverse plane partitions as follows: 1) Write $\Phi \in RPP(w,n)$ as an increasing chain of order ideals $(\phi_1, \dots, \phi_n)$. 2) Apply $t_i$ to each component to get $(t_i(\phi_1), \dots, t_i(\phi_n))$. 3) Write each $t_i(\phi_j)$ as an element of $RPP(w,1)$. 4) Sum the corresponding elements to get an element of $RPP(w,n)$.

\begin{example}
    Let $\Phi$ be the same RPP considered in Example \ref{ex:RPP}. Writing $\Phi$ as an increasing chain of order ideals, applying the toggle $t_3$, and summing gives
    \[\small
    t_3\left (\arraycolsep=3pt\def\arraystretch{0.9}\begin{array}{ccccccc}
	& & 0 & & \\
	& 0 & & 0 & \\
	1 & & 0 & & \\
	& 1 & & & 
	\end{array}, \arraycolsep=3pt\def\arraystretch{0.9}\begin{array}{ccccccc}
	& & 0 & & \\
	& 0 & & 0 & \\
	1 & & 1 & & \\
	& 1 & & & 
	\end{array}, \arraycolsep=3pt\def\arraystretch{0.9}\begin{array}{ccccccc}
	& & 0 & & \\
	& 1 & & 1 & \\
	1 & & 1 & & \\
	& 1 & & & 
	\end{array}			
    \right ) = \left (\arraycolsep=3pt\def\arraystretch{0.9}\begin{array}{ccccccc}
	& & 0 & & \\
	& 0 & & 0 & \\
	1 & & 1 & & \\
	& 1 & & & 
	\end{array}, \arraycolsep=3pt\def\arraystretch{0.9}\begin{array}{ccccccc}
	& & 0 & & \\
	& 0 & & 0 & \\
	1 & & 0 & & \\
	& 1 & & & 
	\end{array}, \arraycolsep=3pt\def\arraystretch{0.9}\begin{array}{ccccccc}
	& & 1 & & \\
	& 1 & & 1 & \\
	1 & & 1 & & \\
	& 1 & & & 
	\end{array}			
    \right ) = \arraycolsep=3pt\def\arraystretch{0.9}\begin{array}{ccccccc}
	& & 1 & & \\
	& 1 & & 1 & \\
	3 & & 2 & & \\
	& 3 & & & 
	\end{array}	
    \]
\end{example}

Via the isomorphism $B(n\lambda) \cong RPP(w_0^J, n)$ in Theorem \ref{thm:RPP Crystal} we can consider how toggles interact with the weight of an RPP.

\begin{proposition}[\cite{Dra}, Lemma 2.40]\label{prop:RPP weight}
    For $w$ and $\lambda$ as above, if $\Phi\in RPP(w,n)$ and $t_i$ is a generator of $Tog(w)$ then $\wt(t_i\Phi) = s_i\wt(\Phi)$.
\end{proposition}

Specializing to the crystal $B(n\varpi_1)$ in type $D_m$, it can be realized as the set $RPP(w,n)$ for $w=s_1\cdots s_{m-2}s_{m-1}s_ms_{m-2}\cdots s_1$. As described in Example \ref{eg:multmincrys}, an alternative model is through generalized SSYT. Using the bijection between the two models, the action of the toggles on tableaux has the following description, analogous to the type A setting for Bender--Knuth involutions. For $1\le i\le m-2$, the toggle $t_i$ swaps the number of $i$ fillings and $i+1$ fillings and swaps the number of $\overline{i}$ fillings and $\overline{i+1}$ fillings. The toggle $t_{m-1}$ swaps the fillings $m-1$ and $m$ and swaps $\overline{m-1}$ and $\overline{m}$ and then replaces pairs of $m$ and $\overline{m}$ with pairs of $m-1$ and $\overline{m-1}$. Similarly $t_m$ interchanges $m-1$ with $\overline{m}$ and $m$ with $\overline{m-1}$ and replaces pairs.

\section{Results}\label{sec:results}

In this section we prove Conjecture \ref{cong:cactus toggle} for the Lie algebra $\g$ of type $D_m$ for the case $B(n\varpi_1)$. For $c\in C_{\g}$ and $t\in Tog(w)$, we denote $c \sim t$ to mean $c$ acts the same as $t$ under the image of the map $C_{\g} \to \prod_n S_{B(n\lambda)}$, where $S_{B(n\lambda)}$ is the symmetric group on the elements of ${B(n\lambda)}$. For integers $1\le k\le m$, define the toggles $r_k$ inductively as follows:
$$r_m = t_m;  \quad r_{m-1} = t_{m-1}; \quad r_k = r_{k+1}t_kr_{k+1}t_kr_{k+1} \text{ for } 1\le k\le m-2.$$

\begin{proposition}\label{prop:Single node}
    The single node cactus generator $c_k$ acts the same as the toggle $r_k$ on the crystal $B(n\varpi_1)$ in type $D_m$, i.e. $c_k\sim r_k$ for all $k=1,\hdots, m$.
\end{proposition}
\begin{proof}
    We proceed by induction, with the base cases being $k = m$ and $k = m-1$, which follow from weight considerations. Observe that since an element $v \in B(n\varpi_1)$ cannot contain both $m$ and $\overline{m}$ in its filling, the $\epsilon_m$ component of $\wt(v)$ determines the number of $m$ or $\overline{m}$ fillings. Hence specifying the weight $\wt(v)$ and the number $N(v)=(\# (m-1)'s +\# \overline{m-1}'s + \# m's +\# \overline{m}'s \text{ in the filling of $v$})$ uniquely determines the number of each filling. The toggles $t_{m}$ and $t_{m-1}$ both preserve the total number of $m-1$, $\overline{m-1}$, $m$, and $\overline{m}$ in the filling of $v$, and likewise for $c_{m}$ and $c_{m-1}$. Since $\wt(t_kv) = s_k\wt(v) = \wt(c_kv)$, we have that $c_{m-1}v = t_{m-1}v$ and $c_mv = t_mv$ as desired.

    For the inductive step, suppose $1\le k\le m-2$. Let $v\in B(n\varpi_1)$ be an arbitrary crystal element. We will compute $c_kv$ and $c_{k+1}t_kc_{k+1}t_kc_{k+1}v$ directly. Representing $v$ as a semistandard tableau with entries $1,\dots, m, \overline{m}, \dots, \overline{1}$, we only need to consider the boxes with entries $k, k+1, k+2, \overline{k+2}, \overline{k+1}, \overline{k}$ as these are the only entries effected by the crystal operators $f_k$ and $f_{k+1}$ and the toggle $t_k$. Suppose $v$ has $a$ many $k$'s (resp. $\bar{a}$ many $\overline{k}$'s), $b$ many $k+1$'s (resp. $\bar{b}$ many $\overline{k+1}$'s), and $c$ many $k+2$'s (resp. $\bar{c}$ many $\overline{k+2}$'s), recorded by the tuple $(a,b,c,\bar{c},\bar{b},\bar{a})$.  We will first compute $c_kv$. Recall that the cactus generator $c_k$ acts by inverting the $(\mathfrak{sl}_2)_k$ chain containing $v$. Let $v^{high}$ and $v^{low}$ denote the highest and lowest weight elements that chain. In each case, let $\bar{\bar{x}} = |x - \bar{x}|$ for $x = a,b,c$. For convenience, we omit $c$ and $\bar{c}$ in the tuple since they are unchanged by $c_k$.  

   \noindent \textbf{Case 1:} $a\le \bar{a}$ and $b\le \bar{b}$. Using the tensor product rule, we can see that $v^{high} = (a,b,\bar{a}+\bar{b},0)$ and $v^{low} = (0,a+b, b, \bar{\bar{b}} + \bar{a})$. Since $v = f_k^{\bar{a}}v^{high}$, we have that $c_kv = e_k^{\bar{a}}v^{low} = (a,b,\bar{\bar{a}} + b, a + \bar{\bar{b}})$.

    \noindent\textbf{Case 2:} $a\le \bar{a}$ and $b > \bar{b}$. We proceed in the same way as in Case 1. We have that $v^{high} = (a + \bar{\bar{b}}, \bar{b}, \bar{a} + \bar{b}, 0)$, $v^{low} = (0, a+b, \bar{b}, \bar{a})$, $v = e_k^a v^{low}$, and therefore $c_kv = f_k^a v^{high} = (a + \bar{\bar{b}}, \bar{b}, \bar{\bar{a}} + \bar{b}, a)$.

    \noindent\textbf{Case 3:} $a > \bar{a}$ and $b\le \bar{b}$. We have that $v^{high} = (a, b, \bar{a} + \bar{b}, 0)$, $v^{low} = (0, a+b, b, \bar{a} + \bar{\bar{b}})$, $v = f_k^{\bar{a}} v^{high}$, and therefore we get $c_kv = e_k^{\bar{a}} v^{low} = (\bar{a}, \bar{\bar{a}} + b, b, \bar{a} + \bar{\bar{b}})$.

    \noindent\textbf{Case 4:} $a > \bar{a}$ and $b > \bar{b}$. We have that $v^{high} = (a + \bar{\bar{b}}, \bar{b}, \bar{a} + \bar{b}, 0)$, $v^{low} = (0, a+b, \bar{b}, \bar{a})$, $v = e_k^a v^{low}$, and therefore we get $c_kv = f_k^a v^{high} = (\bar{a} + \bar{\bar{b}}, \bar{\bar{a}} + \bar{b}, \bar{b}, \bar{a})$.

    Next we compute $c_{k+1}t_kc_{k+1}t_kc_{k+1}v$ and compare to the previous cases. This time we need $8$ cases ($x \leq \bar{x}$ versus $x > \bar{x}$, for $x=a,b,c$) since the relative sizes of $c$ and $\bar{c}$ impact the calculation of the action of $c_{k+1}$. So, we need to consider  We use the previous cases to compute the action of $c_{k+1}$ with $b$ and $c$ playing the roles of $a$ and $b$ respectively. We show one of the cases below, the remaining seven follow analogously. \\

    \noindent\textbf{Case 1.2:} $a\le \bar{a}$, $b\le \bar{b}$, and $c> \bar{c}$. We apply Case 2, Case 1, then Case 3, to get:
    \begin{align*}
        &c_{k+1}t_kc_{k+1}t_kc_{k+1}(a,b,c,\bar{c},\bar{b},\bar{a}) = c_{k+1}t_kc_{k+1}t_k(a, b + \bar{\bar{c}}, \bar{c}, \bar{\bar{b}} + \bar{c}, b, \bar{a}) \\
        &= c_{k+1}t_kc_{k+1}(b + \bar{\bar{c}}, a, \bar{c}, \bar{\bar{b}} + \bar{c}, \bar{a}, b) = c_{k+1}t_k(b + \bar{\bar{c}}, a, \bar{c}, \bar{\bar{a}} + \bar{c}, a + \bar{\bar{b}}, b) \\
        &= c_{k+1}(a, b + \bar{\bar{c}}, \bar{c}, \bar{\bar{a}} + \bar{c}, b, a + \bar{\bar{b}})= (a,b,c,\bar{c}, \bar{\bar{a}} + b, a + \bar{\bar{b}}) = c_kv
    \end{align*}
    
\noindent    In each case we see that $c_{k+1}t_kc_{k+1}t_kc_{k+1}v = c_kv$ which completes the induction. 
\end{proof}

 For the following results, we retain the same notation as in the proof of Proposition \ref{prop:Single node} and directly compute both expressions based on the same cases.
 
\begin{lemma}\label{lemma:cactus toggle commute}
    For $1\le j\le m-2$, we have $c_jt_j = t_jc_j$.
\end{lemma}
\begin{proof}
    Let $v \in B(n\varpi_1)$. The first case is shown below, the other four follow analogously.

    \noindent\textbf{Case 1:} $a\le \bar{a}$ and $b\le \bar{b}$. We compute as follows, using Case 1 from the proof of Proposition \ref{prop:Single node}: $c_jt_jv = c_jt_j(a,b,\bar{b},\bar{a})= c_j(b,a,\bar{a},\bar{b})= (b,a,a + \bar{\bar{b}}, \bar{\bar{a}} + b)= t_jc_jv.$
\end{proof}

\begin{lemma}\label{lemma:two element}
    For $1\le k\le m-2$, and $k=m$, we have 
    \begin{align*} t_kc_{k+1}f_k &\sim f_{k+1}t_kc_{k+1} \quad\quad t_kc_{k+1}e_k \sim e_{k+1}t_kc_{k+1} \\
        f_kc_{k+1}t_k &\sim c_{k+1}t_kf_{k+1} \quad\quad
        e_kc_{k+1}t_k \sim c_{k+1}t_ke_{k+1}\\
        t_{k+1}c_{k+1}f_k &\sim f_{k}t_{k+1}c_{k+1} \quad\quad t_{k+1}c_{k+1}e_k \sim e_kt_{k+1}c_{k+1} \\
        t_{m-2}c_mf_{m-2} &\sim f_mt_{m-2}c_m \quad\quad
        t_{m-2}c_me_{m-2} \sim e_mt_{m-2}c_m \\
        t_mc_mf_{m-2} &\sim f_{m-2}t_mc_m \quad\quad t_mc_me_{m-2} \sim e_{m-2}t_mc_m.
    \end{align*}
\end{lemma}
\begin{proof}
    We retain the same notation as in the proof of Proposition \ref{prop:Single node} and directly compute both expressions based on the same cases. We show Case 1 for the identity $f_{k+1}t_kc_{k+1} \sim t_kc_{k+1}f_k$, the other cases and identities follow analogously.

    \noindent\textbf{Case 1.1:} $a\le \bar{a}$, $b\le \bar{b}$, and $c\le \bar{c}$.
    \begin{align*}
        &f_{k+1}t_kc_{k+1}(a,b,c,\bar{c},\bar{b},\bar{a}) = f_{k+1}t_k(a,b,c,\bar{\bar{b}} + c, b + \bar{\bar{c}}, \bar{a})\\
        &= f_{k+1}(b,a,c,\bar{\bar{b}} + c, \bar{a}, b + \bar{\bar{c}})= (b,a,c,\bar{\bar{b}} + c - 1, \bar{a} + 1, b + \bar{\bar{c}}) \\
        &t_kc_{k+1}f_k(a,b,c,\bar{c},\bar{b},\bar{a}) = t_kc_{k+1}(a,b,c,\bar{c},\bar{b}-1,\bar{a}+1)\\
        &= t_k(a,b,c,\bar{\bar{b}} + c - 1,b + \bar{\bar{c}}, \bar{a} + 1)= (b,a,c, \bar{\bar{b}} + c - 1, \bar{a} + 1, b + \bar{\bar{c}})
    \end{align*}

    We can see that $t_kc_{k+1}f_k \sim f_{k+1}t_kc_{k+1}$ and $f_kc_{k+1}t_k \sim c_{k+1}t_kf_{k+1}$ are equivalent, using the fact that $t_k$ and $c_{k+1}$ are involutions. Likewise $t_kc_{k+1}e_k \sim e_{k+1}t_kc_{k+1}$ and $e_kc_{k+1}t_k \sim c_{k+1}t_ke_{k+1}$ are equivalent. Finally to see that $t_kc_{k+1}f_k \sim f_{k+1}t_kc_{k+1}$ and $e_kc_{k+1}t_k \sim c_{k+1}t_ke_{k+1}$ are equivalent, we use the fact that the raising and lowering operators are partial inverses.
\end{proof}

\begin{lemma}\label{lemma:connected components}
    Let $1\le k\le m-2$. Then $t_k$ preserves the connected components of $B(n\varpi_1)|_{(\mathfrak{sl}_3)_{\{k,k+1\}}}$. Furthermore, $t_{m-2}$ preserves the connected components of $B(n\varpi_1)|_{(\mathfrak{sl}_3)_{\{m-2,m\}}}$.
\end{lemma}
\begin{proof}
    Let $v\in B(n\varpi_1)$. We will check that $v$ and $t_kv$ have the same $\mathfrak{sl}_3$ highest weight. We retain the same notation as in the proof of Proposition \ref{prop:Single node} and compute the $\mathfrak{sl}_3$ highest weight of both expressions. We can check that the highest weight elements of the branched components of $B(n\varpi_1)|_{(\mathfrak{sl}_3)_{\{k,k+1\}}}$ are of the form $(a,0,c,\bar{c},0,0)$ where $c\le \bar{c}$. We consider the same 8 cases as in Proposition \ref{prop:Single node}, and demonstrate the first case below.

    \noindent\textbf{Case 1.1:} $a\le \bar{a}$, $b\le \bar{b}$, and $c\le \bar{c}$. We get highest weight $(a + b,0,c,\bar{a} + \bar{b} + \bar{c},0,0)$:
    \begin{align*}\small
        e_k^b e_{k+1}^{\bar{a} + \bar{b}} e_k^{\bar{a}}v &= e_k^b e_{k+1}^{\bar{a} + \bar{b}} e_k^{\bar{a}}(a,b,c,\bar{c},\bar{b},\bar{a})= e_k^b e_{k+1}^{\bar{a} + \bar{b}} (a,b,c,\bar{c},\bar{a} + \bar{b},0) = e_k^b (a,b,c,\bar{a} + \bar{b} + \bar{c},0,0) \\
        e_k^a e_{k+1}^{\bar{a} + \bar{b}} e_k^{\bar{b}}t_k v &= e_k^a e_{k+1}^{\bar{a} + \bar{b}} e_k^{\bar{b}}(b,a,c,\bar{c},\bar{a},\bar{b})
        = e_k^a e_{k+1}^{\bar{a} + \bar{b}} (b,a,c,\bar{c},\bar{a} + \bar{b},0)= e_k^a (b,a,c,\bar{a} + \bar{b} + \bar{c},0,0)
    \end{align*}

    In each case, $v$ and $t_kv$ have the same $\mathfrak{sl}_3$ highest weight and we conclude that $t_k$ stabilizes the components of the branched crystal $B(n\varpi_1)|_{(\mathfrak{sl}_3)_{\{k,k+1\}}}$.
\end{proof}

\begin{thm}\label{thm:type A subdiagram}
    Let $J = [i,j] = \{i,i+1, \dots, j\}$ be a subdiagram of the Dynkin diagram $I$ of type $D_m$ with either $1\le i<j\le m-1$ or $j = m$. Then 
    \[
    c_J \sim (c_jt_{j-1}\cdots t_i)(c_jt_{j-1}\cdots t_{i+1})\cdots (c_jt_{j-1})c_j.
    \]
\end{thm}
\begin{proof}
     By induction on $|J|$. When $|J| = 2$, we have $\wt(c_jb) = \wt(t_jb) = s_j\wt(b)$, so $\wt(c_jt_{j-1}c_jb) = s_js_{j-1}s_j\wt(b) = w_0^J\wt(b)$. By Lemma \ref{lemma:connected components}, we see that $c_jt_{j-1}c_j$ preserves the connected components of $B(n\varpi_1)|_{(\mathfrak{sl}_3)_{\{j-1,j\}}}$. Using Lemma \ref{lemma:two element}, we have that
     \begin{align*}
         c_jt_{j-1}c_jf_j &= c_jt_{j-1}e_jc_j= e_{j-1}c_jt_{j-1}c_j= e_{\theta_J(j)}c_jt_{j-1}c_j \\
          c_jt_{j-1}c_jf_{j-1} &= c_jf_jt_{j-1}c_j = e_jc_jt_{j-1}c_j= e_{\theta_J(j-1)}c_jt_{j-1}c_j
     \end{align*}
     and likewise for $e_j$ and $e_{j-1}$. Hence $c_J \sim c_jt_{j-1}c_j$ as it satisfies the properties which uniquely characterize the cactus action. We can similarly see that the expression acts correctly on weights and preserves the connected components of the branched crystals. For the inductive step, when $i+2\le k\le j$, we have
    \begin{align*}
         (c_jt_{j-1}\cdots t_i)c_{J\backslash\{i\}}f_k &= c_j\cdots t_ie_{\theta_{J\backslash\{i\}}(k)}c_{J\backslash\{i\}}= c_j\cdots t_{i+j+1-k}t_{i+j-k}e_{i+j+1-k}\cdots c_{J\backslash\{i\}}\\
         &= c_j\cdots t_{i+j+1-k}c_{i+j+1-k}^2t_{i+j-k}e_{i+j+1-k}\cdots c_{J\backslash\{i\}}\\
         &= c_j\cdots t_{i+j+1-k}c_{i+j+1-k}e_{i+j-k}c_{i+j+1-k}t_{i+j-k}\cdots c_{J\backslash\{i\}}\tag{Lemma \ref{lemma:two element}}\\
         &= c_j\cdots e_{i+j-k}t_{i+j+1-k}c_{i+j+1-k}c_{i+j+1-k}t_{i+j-k}\cdots c_{J\backslash\{i\}}\tag{Lemma \ref{lemma:two element}}\\
         &= e_{i+j-k}c_j\cdots t_ic_{J\backslash\{i\}}= e_{\theta_J(k)}c_j\cdots t_ic_{J\backslash\{i\}}.
    \end{align*}
    Next, when $k = i+1$ we have
    \begin{align*}
         (c_jt_{j-1}\cdots t_i)c_{J\backslash\{i\}}f_{i+1} &= c_j\cdots t_ie_{\theta_{J\backslash\{i\}}(i+1)}c_{J\backslash\{i\}}
         = c_jt_{j-1}e_j \cdots c_{J\backslash\{i\}}\\
         &= e_{j-1}c_jt_{j-1} \cdots c_{J\backslash\{i\}} 
         = e_{\theta_J(i+1)}c_j\cdots t_i c_{J\backslash\{i\}}. \tag{Lemma \ref{lemma:two element}}
    \end{align*}
    Finally, when $k = i$, by rearranging and using the inductive hypothesis we have (in step 3 use Lemma \ref{lemma:cactus toggle commute}, in steps 4 and 6 use Lemma \ref{lemma:two element}): 
    \begin{align*}
         &(c_j\cdots t_i)(c_j\cdots t_{i+1})c_{J\backslash \{i,i+1\}}f_i = (c_j\cdots t_i)(c_j\cdots t_{i+1})f_ic_{J\backslash \{i,i+1\}}\\
         &= (c_j\cdots t_i)(c_j\cdots c_{i+1}c_{i+1}t_{i+1})f_ic_{J\backslash \{i,i+1\}}
         = (c_j\cdots t_i)c_j\cdots c_{i+1}t_{i+1}c_{i+1}f_ic_{J\backslash \{i,i+1\}}\\
         &= (c_j\cdots t_i)c_j\cdots t_{i+2}c_{i+1}f_it_{i+1}c_{i+1}c_{J\backslash \{i,i+1\}}= (c_j\cdots t_i)c_j\cdots t_{i+2}t_it_ic_{i+1}f_it_{i+1}c_{i+1}c_{J\backslash \{i,i+1\}}\\
         &= (c_j\cdots t_i)c_j\cdots t_{i+2}t_if_{i+1}t_ic_{i+1}t_{i+1}c_{i+1}c_{J\backslash \{i,i+1\}} = (c_j\cdots t_{i+1})(c_j\cdots t_{i+2})f_{i+1}\cdots c_{J\backslash \{i,i+1\}}\\
         &= c_{J\backslash \{i\}}c_{J\backslash \{i,i+1,i+2\}}f_{i+1}\cdots c_{J\backslash \{i,i+1\}}= c_{J\backslash \{i\}}f_{i+1}c_{J\backslash \{i,i+1,i+2\}}\cdots c_{J\backslash \{i,i+1\}}\\
         &= e_jc_{J\backslash \{i\}}c_{J\backslash \{i,i+1,i+2\}}\cdots c_{J\backslash \{i,i+1\}}= e_j(c_j\cdots t_i)(c_j\cdots t_{i+1})\cdots c_{J\backslash \{i,i+1\}}\\
         &= e_{\theta_J(i)}(c_j\cdots t_i)(c_j\cdots t_{i+1})\cdots c_{J\backslash \{i,i+1\}}
    \end{align*}
    These cases, together with analogous arguments for the $e_i$'s complete the proof.        
\end{proof}

\begin{thm}\label{thm:type D subdiagram}
    Let $1\le j\le m-2$, and set $J = \{j, j+1, \dots, m-2, m-1, m\}$. If $s_{i_1}\cdots s_{i_l}$ is any reduced word for the longest element of the subgroup $W_J$, then $c_J \sim c_{i_1}\cdots c_{i_l}$ in for the crystal $B(n\varpi_1)$.
\end{thm}
\begin{proof}
    We will establish the claim for a particular choice of reduced word and then use the fact that adjacent single node cactus generators braid.

    Let $p_{m-1} = c_{m-1}$, $p_m = c_m$, and $p_i = c_ic_{i+1}\cdots c_{m-2}c_{m-1}c_mc_{m-2} \cdots c_{i+1}c_i$ for $1\le i \le m-2$. Define $q_i = p_ip_{i+1}\cdots p_{m-2}p_{m-1}p_m$ for $1\le i\le m-2$. Then $q_i$ is a reduced expression for the longest element of the Weyl group generated by $s_i, \dots, s_m$.

    We claim that $c_J \sim q_j$. It is immediate that $\wt(q_jb) = w_0^J\wt(b)$ and that $q_j$ preserves the branched components of $B(n\varpi_1)$ as each single node cactus generator does. It remains to check that $q_jf_kb = e_{\theta_J(k)}q_jb$ for all $b\in B(n\varpi_1)$ and all $j\le k\le m$. We will proceed by induction. When $j = m-2$, we have that $q_j = c_{m-2}c_{m-1}c_mc_{m-2}c_{m-1}c_m$. (in step 4 use Prop \ref{prop:Single node}, in steps 3 and 7 use Lemma \ref{lemma:two element}):
    \begin{align*}
        &q_jf_{m-2} = c_{m-2}c_{m-1}c_mc_{m-2}c_{m-1}c_mf_{m-2}= c_{m-2}c_{m-1}c_mc_{m-2}c_{m-1}t_{m-2}t_{m-2}c_mf_{m-2}\\
        &= c_{m-2}c_{m-1}c_mc_{m-2}c_{m-1}t_{m-2}f_{m}t_{m-2}c_m \\
        &= c_{m-2}c_{m-1}c_m(c_{m-1}t_{m-2}c_{m-1}t_{m-2}c_{m-1})c_{m-1}t_{m-2}f_{m}t_{m-2}c_m \\
        &= c_{m-2}c_mt_{m-2}c_{m-1}f_mt_{m-2}c_m= c_{m-2}c_mt_{m-2}f_mc_{m-1}t_{m-2}c_m\\
        &= c_{m-2}f_{m-2}c_mt_{m-2}c_{m-1}t_{m-2}c_m = e_{m-2}c_{m-2}c_mt_{m-2}c_{m-1}t_{m-2}c_m= e_{\theta_J(m-2)}q_j.
    \end{align*}
    Similarly, we have (in steps 4 and 7 use Prop \ref{prop:Single node}, in steps 6 and 9 use Lemma \ref{lemma:two element})
    \begin{align*}
       & q_jf_{m-1} = c_{m-2}c_{m-1}c_mc_{m-2}c_{m-1}c_mf_{m-1}= c_{m-2}c_{m-1}c_mc_{m-2}c_{m-1}f_{m-1}c_m\\
        &= c_{m-2}c_{m-1}c_mc_{m-2}e_{m-1}c_{m-1}c_m= c_{m-2}c_{m-1}c_m(c_{m-1}t_{m-2}c_{m-1}t_{m-2}c_{m-1})e_{m-1}c_{m-1}c_m \\
        &= c_{m-2}c_mt_{m-2}c_{m-1}t_{m-2}f_{m-1}c_{m-1}c_{m-1}c_m= c_{m-2}c_mt_{m-2}f_{m-2}c_{m-1}t_{m-2}c_m \\
        &= (c_{m}t_{m-2}c_{m}t_{m-2}c_{m})c_mt_{m-2}f_{m-2}c_{m-1}t_{m-2}c_m = c_m t_{m-2}c_m f_{m-2}c_{m-1}t_{m-2}c_m \\
        &= c_m f_m t_{m-2}c_m c_{m-1}t_{m-2}c_m = e_m c_m t_{m-2}c_m c_{m-1}t_{m-2}c_m= e_{\theta_J(m-1)}q_j.
    \end{align*}
    and analogously $q_jf_m = e_{\theta_J(m)}q_j$. The argument for the $e$'s are identical, which establishes the base case. Suppose $1\le j\le m-3$ and that $c_{J\backslash \{j\}} \sim q_{j+1}$. If $j<k\le m-2$, then we have (in steps 5 and 9 use Prop \ref{prop:Single node}, in steps 7 and 11 use Lemma \ref{lemma:two element}):
    \begin{align*}
        q_jf_k &= p_jq_{j+1}f_k= p_je_kq_{j+1}= c_j\cdots c_je_kq_{j+1}= c_j \cdots c_{k-1}e_k\cdots c_jq_{j+1}\\
        &= c_j \cdots c_k(c_kt_{k-1}c_kt_{k-1}c_k)e_k \cdots c_j q_{j+1}= c_j \cdots t_{k-1}c_kt_{k-1}f_kc_k \cdots c_jq_{j+1}\\
        &= c_j \cdots t_{k-1}f_{k-1}c_kt_{k-1}\cdots c_jq_{j+1} = c_j \cdots c_{k-1}c_kt_{k-1}f_{k-1} \cdots c_jq_{j+1}\\
        &= c_j \cdots (c_kt_{k-1}c_kt_{k-1}c_k)c_kt_{k-1}f_{k-1} \cdots c_jq_{j+1} = c_j \cdots c_kt_{k-1}c_kf_{k-1} \cdots c_jq_{j+1}\\
        &= c_j \cdots c_kf_kt_{k-1}c_k \cdots c_jq_{j+1} = c_j \cdots e_kc_kt_{k-1}\cdots c_jq_{j+1}= e_kp_jq_{j+1}= e_{\theta_J(k)}q_j.
    \end{align*}
    For the spin nodes, it suffices to check that $p_jf_{m-1} = f_mp_j$ and $p_jf_m = f_{m-1}p_j$. We have (in step 2 use Proposition \ref{prop:Single node}, in steps 5 and 6 use Lemma \ref{lemma:two element}):
    \begin{align*}
        p_jf_{m-1} &= c_j \cdots c_jf_{m-1}= c_j \cdots c_{m-2}c_mc_{m-1}c_{m-2}f_{m-1}\cdots c_j\\
        &= c_j \cdots (c_{m}t_{m-2}c_{m}t_{m-2}c_{m})c_mc_{m-1}(c_{m-1}t_{m-2}c_{m-1}t_{m-2}c_{m-1})f_{m-1} \cdots c_j \\
        &= c_j \cdots c_mt_{m-2}c_mc_{m-1}t_{m-2}c_{m-1}f_{m-1} \cdots c_j= c_j \cdots c_mc_{m-1}t_{m-2}e_{m-1}c_{m-1} \cdots c_j\\
        &= c_j \cdots c_mt_{m-2}c_me_{m-2}c_{m-1}\cdots c_j= c_j \cdots c_me_mt_{m-2}c_mc_{m-1} \cdots c_j\\
        &= c_j \cdots f_mc_mt_{m-2}c_mc_{m-1}t_{m-2}c_{m-1} \cdots c_j= f_m p_j.
    \end{align*}
    and likewise $p_jf_m = f_{m-1}p_j$. Lastly, when $j = k$, we have (in step 5 use Proposition \ref{prop:Single node}, in steps 4 and 7 use Lemma \ref{lemma:two element}):
    \begin{align*}
        q_jf_j &= p_jp_{j+1}q_{j+2}f_j= c_jc_{j+1}\cdots c_{j+1}c_jc_{j+1}\cdots c_{j+1}f_jq_{j+2}\\
        &= c_jc_{j+1}t_j\cdots t_jc_{j+1}c_jc_{j+1}t_j \cdots t_jc_{j+1}f_jq_{j+2}\\
        &= c_jc_{j+1}t_jc_{j+2}\cdots c_{j+2}t_jc_{j+1}c_jc_{j+1}t_jc_{j+2} \cdots c_{j+2}f_{j+1}t_jc_{j+1}q_{j+2}\\
        &= c_jc_{j+1}t_jc_{j+2}\cdots c_{j+2}c_{j+1}c_{j+2} \cdots c_{j+2}f_{j+1}t_jc_{j+1}q_{j+2}\\
        &= c_jc_{j+1}t_jp_{j+2}c_{j+1}p_{j+2}f_{j+1}t_jc_{j+1}q_{j+2}= c_jc_{j+1}t_jf_{j+1}p_{j+2}c_{j+1}p_{j+2}t_jc_{j+1}q_{j+2}\\
        &= c_jf_jc_{j+1}\cdots c_{j+1}c_jc_{j+1}\cdots c_{j+1}q_{j+2}= e_jp_jp_{j+1}q_{j+2}= e_{\theta_J(j)}q_j
    \end{align*}
    which together with an analogous argument for $e_j$ completes the claim.
\end{proof}

As a consequence of the results above, we get the following corollary. 

\begin{corollary}
    The action of the cactus group $C_{\mathfrak{g}}$ on the crystal $B(n\varpi_1)$ in type $D_m$ is generated by elements corresponding to length 1 and 2 diagrams. This answers Conjecture \ref{cong:cactus toggle}, Part 2.
\end{corollary}

\end{document}